\documentclass[11pt]{article}

\usepackage{amsmath,amssymb}
\usepackage{latexsym}
\usepackage{graphicx}
\usepackage{bm}

\newtheorem{thm}{Theorem}
\newtheorem{cor}[thm]{Corollary}

\newtheorem{lem}[thm]{Lemma}
\newtheorem{rem}[thm]{Remark}

\newenvironment{proof}{\begin{trivlist}
                       \item[]{\bf Proof.}
                       \hspace{0cm}}{\hfill $\Box$
                       \end{trivlist}}

\begin{document}
\title{A nonlinear inequality and applications}

\author{N. S. Hoang$\dag$\footnotemark[1]\quad and\quad 
A. G. Ramm$\dag$\footnotemark[3]
\\
\\
$\dag$Mathematics Department, Kansas State University,\\
Manhattan, KS 66506-2602, USA
}

\renewcommand{\thefootnote}{\fnsymbol{footnote}}
\footnotetext[1]{Email: nguyenhs@math.ksu.edu}
\footnotetext[3]{Corresponding author. Email: ramm@math.ksu.edu}
\date{}
\maketitle

\begin{abstract} \noindent 
A nonlinear inequality is formulated in the paper.
An estimate of the rate of decay of solutions to this inequality is 
obtained.
This inequality is of interest in a study of dynamical systems 
and nonlinear evolution equations. It can be applied to
the study of global existence of solutions to nonlinear PDE.


\end{abstract}

\section{Introduction}

In this paper the following differential inequality 
\begin{equation}
\label{neq1} 
\dot{g}(t) \le -\gamma(t)g(t) + \alpha(t)g^p(t) +
\beta(t),\qquad t\ge t_0,\quad p>1, \quad \dot{g}=\frac{dg}{dt},
\end{equation} 
where $g(t)\ge 0$, 
is studied.  In equation \eqref{neq1}, $\alpha(t),\beta(t)$ and 
$\gamma(t)$ are continuous
functions, defined on $[t_0,\infty)$, where $t_0\ge 0$ is a 
fixed
number and $\alpha(t)\ge 0,\forall t\ge t_0$. 
Under the assumptions of Theorem~\ref{thm1} (see below)
 we prove that there exist solutions to inequality \eqref{neq1}
and all solutions to inequality \eqref{neq1} are defined for all $t\ge t_0$. 
Estimates of the rate of decay of solutions to this inequality
are obtained and formulated in \eqref{3eq10} and \eqref{hehe}. These new results can be used in a study of dynamical 
systems
and nonlinear evolution equations.  For example, inequality \eqref{neq1}
is used in  Section~\ref{sec3} in a study of the Dynamical Systems 
Method (DSM) for solving 
nonlinear equations of the type $F(u)=f$, where $F:H\to H$ is a monotone 
operator, 
and $H$ 
is a Hilbert space.   The DSM we study
in Section \ref{sec3} is continuous analog of the  regularized 
Newton's method
for solving equations with monotone operators. 
The local boundedness of the second Fr\'{e}chet derivative of $F$ was
assumed earlier in a study of a similar method, inequality \eqref{neq1} with $p=2$ 
was used, and
an estimate for the decay of $g(t)$ as $t\to \infty$ was derived  
with the use of a comparison lemma and a
closed form solution to a special Riccati's equation (see \cite{R499}).  
The argument from \cite{R499} 
is not possible to extend  to the case $p\neq 2$.
The estimate of solutions to  inequality \eqref{neq1} with $p=2$
was also used in  \cite{R549} in a study of a DSM for solving
ill-posed operator equations.  

In this paper sufficient conditions on $\alpha,\beta$ and $\gamma$
are found, which 
yield the global existence and an estimate of the rate of decay of solutions to
\eqref{neq1}. The method of the proof of these  results 
is different from that in \cite{R499}.
It does not require the knowledge of a closed 
form solution of a differential equation. Discrete analogs of 
the inequality \eqref{neq1} are also found (see Theorems~\ref{lem1} and
\ref{cor1}).  
These new results can be applied to the study of the global
existence of solutions to nonlinear PDE.

The paper is organized as follows:  In Section \ref{sec2}, the main
results, namely, Theorems \ref{thm1}, \ref{thm2}, \ref{lem1}, \ref{cor1} and \ref{cor4} are formulated and
proved. An upper bound for $g(t)$ is obtained under some conditions on
$\alpha,\beta,\gamma$. This upper bound gives a sufficient condition
for the relation $\lim_{t\to \infty}g(t)=0$ to hold, and also
gives a rate of decay of $g(t)$ as $t\to \infty$.
 In Section~\ref{sec3} a version of the
DSM is studied. The main result in this Section is Theorem \ref{thm8}. 
In its proof an application of Theorem~\ref{thm1}  is essential. 

\section{Main results}
\label{sec2}

\begin{thm}
\label{thm1}
Let $\alpha(t),\beta(t)$ and $\gamma(t)$ be continuous 
functions on $[t_0,\infty)$ and $\alpha(t)\ge0,\forall t\ge t_0$.
Suppose there exists a function $\mu(t)>0$, $\mu\in C^1[t_0,\infty)$,
such that
\begin{align}
\label{1eq4}
\frac{\alpha(t)}{\mu^p(t)}+ \beta(t) &\le \frac{1}{\mu(t)}\bigg{[}\gamma -
\frac{\dot{\mu}(t)}{\mu(t)}\bigg{]}.
\end{align}
Let $g(t)\ge 0$ be a solution to inequality \eqref{neq1} such that
\begin{equation}
\label{1eq6}
\mu(t_0)g(t_0)    < 1.
\end{equation}
Then $g(t)$ exists globally and the following estimate holds:
\begin{equation}
\label{3eq10}
0\le g(t) < \frac{1}{\mu(t)},\qquad \forall t\ge t_0.
\end{equation}
Consequently, if $\lim_{t\to\infty} \mu(t)=\infty$, then
\begin{equation}
\lim_{t\to\infty} g(t)= 0.
\end{equation}
\end{thm}

\begin{proof}
Denote $w(t):=g(t)e^{\int_{t_0}^t\gamma(s)ds}$. Then inequality 
\eqref{neq1} takes the form
\begin{equation}
\label{eq6}
\dot{w}(t) \le a(t)w^p(t) + b(t),\qquad w(t_0)=g(t_0):= g_0,
\end{equation}
where
\begin{equation}
a(t):=\alpha(t) e^{(1-p)\int_{t_0}^t\gamma(s)ds},
\qquad b(t):=\beta(t) e^{\int_{t_0}^t\gamma(s)ds}.
\end{equation}
Denote 
\begin{equation}
\label{eq8}
\eta(t) = \frac{e^{\int_{t_0}^t\gamma(s)ds}}{\mu(t)}.
\end{equation}
From inequality \eqref{1eq6} and relation \eqref{eq8} one gets
\begin{equation}
\label{eq9}
w(t_0)=g(t_0) < \frac{1}{\mu(t_0)}=\eta(t_0).
\end{equation}
It follows from the inequalities \eqref{1eq4}, \eqref{eq6} 
and \eqref{eq9} that
\begin{equation}
\label{eq10}
\begin{split}
\dot{w}(t_0) &\le \alpha(t_0)\frac{1}{\mu^{p}(t_0)}
+ \beta(t_0)
\le \frac{1}{\mu(t_0)}\bigg{[}\gamma -
\frac{\dot{\mu}(t_0)}{\mu(t_0)}\bigg{]}
 = \frac{d}{dt} \frac{e^{\int_{t_0}^t\gamma(s)ds}}
{\mu(t)}\bigg{|}_{t=t_0} = \dot{\eta}(t_0).
\end{split}
\end{equation}
From the inequalities \eqref{eq9} and \eqref{eq10} it follows 
that there exists $\delta>0$ such that
\begin{equation}
w(t) < \eta(t),\qquad  t_0\le t \le t_0 + \delta.
\end{equation}
To continue the proof we need two Claims.

{\it Claim 1.}  {\it If} 
\begin{equation}
w(t) \le \eta(t),\qquad \forall t \in[ t_0, T],\quad T>t_0,
\end{equation}
{\it then} 
\begin{equation}
\dot{w}(t) \le \dot{\eta}(t),\qquad  \forall t \in[ t_0, T].
\end{equation}
{\it Proof of Claim 1.}
 
It follows from inequalities \eqref{1eq4}, \eqref{eq6} 
and the inequlity $w(T)\leq \eta(T)$,  that
\begin{equation}
\label{eq14}
\begin{split}
\dot{w}(t) &\le e^{(1-p)\int_{t_0}^t\gamma(s)ds}\alpha(t)
\frac{e^{p\int_{t_0}^t\gamma(s)ds}}{\mu^{p}(t)}
+ \beta(t)e^{\int_{t_0}^t\gamma(s)ds}\\
&\le \frac{e^{\int_{t_0}^t\gamma(s)ds}}{\mu(t)}
\bigg{[}\gamma -\frac{\dot{\mu}(t)}{\mu(t)}\bigg{]}\\
& = \frac{d}{dt} \frac{e^{\int_{t_0}^t\gamma(s)ds}}
{\mu(t)}\bigg{|}_{t=t} = \dot{\eta}(t),\qquad \forall t\in [t_0,T].
\end{split}
\end{equation}
{\it Claim 1} is proved. $\hfill$ $\Box$

Denote 
\begin{equation}
\label{eq12}
T:=\sup \{\delta \in\mathbb{R}^+: w(t) < \eta(t),\, 
\forall t \in[t_0, t_0 + \delta]\}.
\end{equation}
{\it Claim 2.}  {\it One has $T=\infty$.} 

Claim 2 guarantees the existence of a nonnegative solution to inequality \eqref{neq1}
for all $t\ge t_0$. Also, it guarantees the global existence for any solution to 
inequality \eqref{neq1} satisfying \eqref{1eq6}. 

{\it Proof of Claim 2.}

Assume the contrary, i.e.,  $T<\infty$. 
From the definition of $T$ and the continuity of $w$ and $\eta$ one 
gets
\begin{equation}
\label{eq13}
w(T) \le \eta(T).
\end{equation}
It follows from 
inequality \eqref{eq13} and {\it Claim 1} that 
\begin{equation}
\label{2eq18}
\dot{w}(t)\le \dot{\eta}(t),\qquad \forall t\in [t_0,T].
\end{equation}
This implies 
\begin{equation}
\label{2eq19}
w(T)-w(t_0) = \int_{t_0}^T \dot{w}(s)ds \le \int_{t_0}^T \dot{\eta}(s)ds 
= \eta(T)-\eta(t_0).
\end{equation}
Since $w(t_0)<\eta(t_0)$ by assumption \eqref{1eq6}, it follows from inequality \eqref{2eq19} that
\begin{equation}
\label{eq20}
w(T) < \eta(T).
\end{equation}
Inequality \eqref{eq20} and inequality \eqref{2eq18} with $t=T$ imply that
 there exists an $\epsilon>0$ such that
\begin{equation}
w(t) < \eta(t),\qquad  T\le t \le T + \epsilon.
\end{equation}
This contradicts the definition of $T$ in \eqref{eq12}, and the 
contradiction proves the desired conclusion $T=\infty$. 

Claim 2 is proved. $\hfill$ $\Box$ 

It follows from the definitions of $\eta(t)$ and $w(t)$ and from 
the relation $T=\infty$ that
\begin{equation}
g(t) = e^{-\int_{t_0}^t\gamma(s)ds} w(t)< 
e^{-\int_{t_0}^t\gamma(s)ds}\eta(t) = \frac{1}{\mu(t)},\qquad \forall t> t_0.
\end{equation}
Theorem~\ref{thm1} is proved.
\end{proof}

\begin{thm}
\label{thm2}
Let $\alpha, \beta,\gamma$ and $g$ be as in Theorem~\ref{thm1}, 
$\mu(t)>0$, $\mu\in C^1[t_0,\infty)$, and
let conditions \eqref{1eq4} holds. 
Assume also that
\begin{equation}
\label{hihi}
g(t_0)\mu(t_0) \le 1.
\end{equation}
Then the following inequality holds:
\begin{equation}
\label{hehe}
g(t)\leq \frac{1}{\mu(t)},\qquad \forall t\ge t_0.
\end{equation}
\end{thm}

\begin{proof}
By Theorem~\ref{thm1} 
it suffices to assume $g(t_0)\mu(t_0)=1$. 
Make the substitutions $w(t):=g(t)e^{\int_{t_0}^t\gamma(s)ds}$ as in Theorem~\ref{thm1} 
and get inequality \eqref{eq6}, where $a$ and $b$ are nonnegative 
functions defined in (8). Denote $\eta(t):=\frac{e^{\int_{t_0}^t\gamma(s)ds}}{\mu(t)}$.
The equality $g(t_0)\mu(t_0)=1$ implies $w(t_0)=\eta(t_0)$. 

Define 
\begin{equation}
T=\sup \{t_1: w(t)\le\eta(t),\, \forall t\in [t_0,t_1]\}.
\end{equation}
{\it Let us show that $T=\infty$.}

Assume the contrary, i.e., $T<\infty$. Then one has 
$w(t)\le \eta(t),\, \forall t\in [t_0,T]$, $w(T)=\eta(T)$, and 
$w(T+\delta)>\eta(T+\delta)$
for some sufficiently small $\delta\in(0,\delta_1)$,\, for any fixed $\delta_1>0$. 
Let $h_\epsilon(t)$ solve the following problem:
\begin{equation}
\label{forh}
\dot{h}_\epsilon(t) = - \gamma(t)h_\epsilon(t) + \alpha(t)h^p_\epsilon(t) + \beta(t),
\qquad h_\epsilon(T):=g(T)-\epsilon,\quad t\ge T,
\end{equation}
where $\epsilon>0$ is sufficiently small. 
Let 
$$
k_\epsilon(t):=h_\epsilon(t)e^{\int_{t_0}^t\gamma(s)ds}.
$$ 
Equation \eqref{forh} implies
\begin{equation}
\label{req27}
\dot{k}_\epsilon(t)= a(t)k^p_\epsilon(t) + b(t).
\end{equation}  
It follows from Theorem~\ref{thm1} with $g:=h_\epsilon$ that
\begin{equation}
0\le h_\epsilon(t)< \frac{1}{\mu(t)},\qquad 
0\le k_\epsilon(t)< \eta(t)\qquad \forall \epsilon>0,\quad 
\forall t\ge T.
\end{equation}
A comparison lemma (\cite[p.99]{R499}) implies $w(t)\leq k_0(t),\forall t\ge t_0$.
Note that there exist $\delta_1>0$ and $M>0$ such that $k_0(t)\le M,\forall t\in [T,T+\delta_1]$. 
Thus,  
we have the inequality: 
\begin{equation}
\label{req28}
0\le k_\epsilon(T+\delta)< \eta(T+\delta)< w(T+\delta)\le k_0(T+\delta)\le M.
\end{equation}
Define $v_\epsilon(t):=k_0(t)-k_\epsilon(t)$. 
From inequality \eqref{eq6} and equation \eqref{req27} one gets the following inequality
\begin{equation}
\label{req29}
\begin{split}
\dot{v}_\epsilon(t) = a(t)\big{[}k_0^p(t) - k^p_\epsilon(t) \big{]}\le C v_\epsilon(t),
\qquad \forall t\in [T,T+\delta],
\end{split}
\end{equation}
where $C:=C(p,a,M):=\|a\|_{C[T,T+\delta]}M^{p-1}$. 
Here an elementary inequality 
$$
|x^p-y^p|\le p\max(|x|,|y|)^{p-1}|x-y|,\qquad \forall x,y\in \mathbb{R},\quad p\ge 1,
$$
was used.
Inequality \eqref{req29} implies
\begin{equation}
\label{eq30}
0\le v_\epsilon(t)\le e^{(t-T)C}v_\epsilon(T) \le e^{\delta C}\epsilon  e^{\int_{t_0}^T\gamma(s)ds},
\qquad \forall t\in [T,T+\delta].
\end{equation}
Here, we have used the formula $v_\epsilon(T)=\epsilon e^{\int_{t_0}^T\gamma(s)ds}$.
It follows from \eqref{eq30} and \eqref{req28} that
\begin{equation}
\label{hichic}
0< w(T+\delta) -\eta(T+\delta)\le \lim_{\epsilon\to 0} v_\epsilon(T+\delta) = 0.
\end{equation}
This contradiction implies that $T=\infty$.
Theorem \eqref{thm2} is proved. 
\end{proof}

\begin{cor}
\label{corrol2}
Let $\alpha(t),\beta(t)$ and $\gamma(t)$ are continuous 
functions on $[t_0,\infty)$ and $\alpha(t)\ge 0$, $\forall t\ge t_0$.
Suppose there exists a function $\mu(t)>0$, $\mu\in C^1[t_0,\infty)$,
such that
\begin{align}
\label{1eq4x}
0\le \alpha(t)&\le \theta \mu^{p-1}\bigg{[}\gamma -
\frac{\dot{\mu}(t)}{\mu(t)}\bigg{]},
\qquad \dot{u}:=\frac{du}{dt},\quad \theta=const \in (0,1),\\
\label{1eq5x}
\beta(t)      &\le \frac{1-\theta}{\mu}\bigg{[}\gamma -
\frac{\dot{\mu}(t)}{\mu(t)}\bigg{]}.
\end{align}
Let $g(t)\ge0$ be a solution to inequality \eqref{neq1} such that 
\begin{equation}
\label{1eq6x}
\mu(t_0)g(t_0)    \le 1.
\end{equation}
Then $g(t)$ exists globally and the following estimate holds:  
 \begin{equation}
\label{3eq10x}
0\le g(t) \le \frac{1}{\mu(t)},\qquad \forall t\ge 0.
\end{equation}
Consequently, if $\lim_{t\to\infty} \mu(t)=\infty$, then
\begin{equation}
\lim_{t\to\infty} g(t)= 0.
\end{equation}
\end{cor}

Let us consider a {\it discrete analog} of Theorem \ref{thm1}.
We wish to study the following inequality: 
$$
\frac{g_{n+1}-g_n}{h_n}\le -\gamma_n g_n+\alpha_n g_n^p +
\beta_n,\qquad h_n > 0,\quad 0< h_n\gamma_n < 1,\quad p>1,
$$
and the inequality:
$$
g_{n+1} \le (1-\gamma_n) g_n + \alpha_n g_n^p + \beta_n,\quad n\ge 0,
\qquad 0<\gamma_n<1,\quad p>1,
$$
where $g_n, \beta_n, \gamma_n$ and $\alpha_n$ are positive sequences
of real numbers.
Under suitable assumptions on $\alpha_n, \beta_n$ and $\gamma_n$, 
we obtain an upper bound for $g_n$ as $n\to\infty$. In particular, we give 
sufficient conditions for $\lim_{n\to\infty}g_n=0$, and estimate the rate 
of decay of $g_n$ as $n\to\infty$. This result can be used in a study of 
evolution problems, 
in a study of iterative processes, and in a study of nonlinear PDE.

\begin{thm}
\label{lem1}
Let $\alpha_n,\gamma_n$ and $g_n$ be nonnegative sequences 
of numbers, and the following inequality holds:
\begin{equation}
\label{eq1}
\begin{split}
\frac{g_{n+1}-g_n}{h_n}&\le -\gamma_n g_n+
\alpha_n g_n^p +\beta_n,\qquad h_n > 0,\quad 0< h_n\gamma_n < 1,
\end{split}
\end{equation}
or, equivalently, 
\begin{equation}
\qquad g_{n+1}\le g_n(1-h_n\gamma_n) +
\alpha_n h_n g_n^p+h_n\beta_n,\qquad h_n > 0,\quad 0< h_n\gamma_n < 1.
\end{equation}
If there is a monotonically growing sequence of positive numbers 
$(\mu_n)_{n=1}^\infty$, such that the following conditions hold:
\begin{align}
\label{eq3}
\frac{\alpha_n}{\mu_n^p}+\beta_n&\le \frac{1}{\mu_n}\bigg{(}\gamma_n -
\frac{\mu_{n+1}-\mu_n}{\mu_n h_n}\bigg{)},\\
\label{eq2}
g_0&\le\frac{1}{\mu_0},
\end{align}
then
\begin{equation}
\label{eq5}
g_n\le\frac{1}{\mu_n} \qquad \forall n\ge 0.
\end{equation}
Therefore, if $\lim_{n\to\infty}\mu_n =\infty$, then $\lim_{n\to\infty} 
g_n = 0$.
\end{thm}

\begin{proof}
Let us prove \eqref{eq5} by induction. Inequality \eqref{eq5} holds for $n=0$ by assumption \eqref{eq2}. 
Suppose that \eqref{eq5} holds for all $n\le m$. From inequalities 
\eqref{eq1}, \eqref{eq3}, and from the induction 
hypothesis 
$g_n\le\frac{1}{\mu_n}$, $n\le m$, one gets
\begin{equation}
\begin{split}
g_{m+1}&\le g_m(1-h_m\gamma_m) + \alpha_m h_m g_m^p + h_m\beta_m\\
\le& \frac{1}{\mu_m}(1-h_m\gamma_m)+ h_m\frac{\alpha_m}{\mu_m^p}+ h_m\beta_m\\
\le& \frac{1}{\mu_m}(1-h_m\gamma_m) + \frac{h_m}{\mu_m}\bigg{(}\gamma_m -
\frac{\mu_{m+1}-\mu_m}{\mu_m h_m}\bigg{)}\\
=& \frac{1}{\mu_m}-\frac{\mu_{m+1}-\mu_m}{\mu_m^2}\\
=& \frac{1}{\mu_{m+1}}- (\mu_{m+1}-\mu_m)\big{(}\frac{1}{\mu_m^2} - 
\frac{1}{\mu_m \mu_{m+1}} \big{)}\\
=& \frac{1}{\mu_{m+1}}- \frac{(\mu_{m+1}-\mu_m)^2}{\mu_n^2 \mu_{m+1}} 
\le\frac{1}{\mu_{m+1}}.
\end{split}
\end{equation}
Therefore, inequality \eqref{eq5} holds for $n=m+1$. 
Thus, inequality \eqref{eq5} holds for all $n\ge 0$ by induction. 
Theorem~\ref{lem1} is proved.
\end{proof}

\begin{cor}
\label{corrol3}
Let $\alpha_n,\gamma_n$ and $g_n$ be nonnegative sequences 
of numbers, and the following inequality holds:
\begin{equation}
\qquad g_{n+1}\le g_n(1-h_n\gamma_n) +
\alpha_n h_n g_n^p+h_n\beta_n,\qquad h_n > 0,\quad 0< h_n\gamma_n < 1.
\end{equation}
If there is a monotonically growing sequence of positive numbers 
$(\mu_n)_{n=1}^\infty$, such that the following conditions hold:
\begin{align}
\label{veq3}
\alpha_n&\le\theta\mu_n^{p-1}\bigg{(}\gamma_n -
\frac{\mu_{n+1}-\mu_n}{\mu_n h_n}\bigg{)},\qquad \theta=const\in (0,1),\\
\label{veq4}
\beta_n&\le\frac{1-\theta}{\mu_n}\bigg{(}\gamma_n -
\frac{\mu_{n+1}-\mu_n}{\mu_n h_n}\bigg{)},\\
\label{veq2}
g_0&\le\frac{1}{\mu_0},
\end{align}
then
\begin{equation}
\label{veq5}
g_n\le\frac{1}{\mu_n} \qquad \forall n\ge 0.
\end{equation}
Therefore, if $\lim_{n\to\infty}\mu_n =\infty$, then $\lim_{n\to\infty} 
g_n = 0$.
\end{cor}

Setting $h_n =1$ in Theorem~\ref{lem1}, one obtains the following result:
\begin{thm}
\label{cor1}
Let $\alpha,\beta,\gamma_n$ and $g_n$ be sequences of nonnegative 
numbers, and
\begin{equation}
\label{2eq1}
\begin{split}
g_{n+1}&\le g_n(1-\gamma_n) +\alpha_n  g_n^p+\beta_n,\qquad 0<\gamma_n <1.
\end{split}
\end{equation}
If there is a monotonically growing sequence $(\mu_n)_{n=1}^\infty>0$ such 
that the following conditions hold
\begin{align}
g_0\le\frac{1}{\mu_0},\qquad\frac{\alpha_n}{\mu_n^p}+\beta_n&\le \frac{1}{\mu_n}\bigg{(}\gamma_n -
\frac{\mu_{n+1}-\mu_n}{\mu_n h_n}\bigg{)},\qquad \forall n\ge 0,
\end{align}
then
\begin{equation}
\label{2eq5}
g_n\le\frac{1}{\mu_n}, \qquad \forall n\ge 0.
\end{equation}
\end{thm}

We have the following corollary
\begin{thm}
\label{cor4}
Let $\alpha,\beta,\gamma_n$ and $g_n$ be sequences of nonnegative 
numbers, and
\begin{equation}
\label{beq1}
\begin{split}
g_{n+1}&\le g_n(1-\gamma_n) +\alpha_n  g_n^p+\beta_n,\qquad 0<\gamma_n <1.
\end{split}
\end{equation}
If there is a monotonically growing sequence $(\mu_n)_{n=1}^\infty>0$ such 
that the following conditions hold
\begin{align}
\label{beq3}
\alpha_n&\le\theta \mu_n^{p-1}\bigg{(}\gamma_n -\frac{\mu_{n+1}-
\mu_n}{\mu_n }\bigg{)},\qquad \theta=const\in (0,1),\\
\label{beq4}
\beta_n&\le\frac{1-\theta}{\mu_n}\bigg{(}\gamma_n -
\frac{\mu_{n+1}-\mu_n}{\mu_n }\bigg{)},\\
\label{beq2}
g_0&\le\frac{1}{\mu_0},
\end{align}
then
\begin{equation}
\label{beq5}
g_n\le\frac{1}{\mu_n}, \qquad \forall n\ge 0.
\end{equation}
\end{thm}

\section{Applications}
\label{sec3}

Let $F:H\to H$ be a Fr\'{e}chet-differentiable map in a real Hilbert 
space $H$.
Assume that
\begin{equation}
\label{eqrr2}
\sup_{u\in B(u_0,R)} \|F'(u)\| \le M = M(u_0,R),
\end{equation}
where $M$ is a constant, $B(u_0,R):=\{u: \|u-u_0\|\le R \}$, $u_0\in H$ is 
some element, 
$R>0$, and there is no
restriction on the growth of $M(R)$ as $R\to\infty$, i.e., 
arbitrary fast growing nonlinearities $F$ are admissible. 

Consider the equation:
\begin{equation}
\label{eqrr3}
F(v)=f
\end{equation}
in $H$.
Assume that $F'(\cdot)\ge 0$, that is, $F$ is monotone: 
\begin{equation}
\langle F(u)-F(v),u-v\rangle \ge 0\qquad \forall u,v\in H.
\end{equation} 
Let equation 
\eqref{eqrr3} have a solution, possibly non-unique, and denote by  
$y$ be the 
unique minimal-norm solution to \eqref{eqrr3}.
If $F$ is monotone and continuous, then $\mathcal{N}_f:=\{ u:F(u)=f\}$ 
is a closed convex set in $H$ (see, e.g., \cite{R499}).
Such a set in a Hilbert space has a unique minimal-norm element. So, the 
solution $y$ is well defined.

Let us assume that
\begin{equation}
\label{8eq2}
\|F(v)-F(u) - F'(u)(u-v)\|\leq M_p(R)\|u-v\|^p,\qquad \forall u,v\in B(0,R),
\end{equation}
where $1<p<2$. For brevity let us denote $M_p:=M_p(R)$.

The main result in this Section is the following Theorem:

\begin{thm}
\label{thm8}
Assume that $F$ is a monotone operator satisfying conditions 
\eqref{eqrr2} and \eqref{8eq2}, 
that equation \eqref{eqrr3} has a solution, and $y$ is its minimal-norm 
solution. 
Let $u_0$ be an arbitrary element in $H$.
Assume that $a(t) = \frac{d}{(c+t)^b}$, where $0<b\le p-1$, $c\ge 
\max \bigg{(}1,\, \frac{2b}{p-1} \bigg{)}$,
and $d>0$ is sufficiently large, so that condition \eqref{eq473} 
holds (see below). 
Let $u(t)$ be the solution to the following DSM:
\begin{equation}
\label{eq46}
\dot{u} = -A^{-1}_{a(t)}[F(u)+a(t)u - f],\qquad u(0)=u_0,
\end{equation}
where $A:=F'(u(t))$ and $A_a:=A+aI$.
Then
\begin{equation}
\lim_{t\to\infty} \|u(t)-y\| = 0.
\end{equation}
\end{thm}

Let us recall the following result (see \cite{R499}, p.112):
\begin{lem}
\label{rejectedlem}
Assume that equation \eqref{eqrr3} is solvable, $y$ is its minimal-norm 
solution, and the operator $F$ is monotone and continuous.
 Then
$$
\lim_{a\to 0} \|V_{a}-y\| = 0,
$$
where $V_{a}$ solves the equation
\begin{equation}
\label{eq40}
F(V_a) + a V_a - f =0,
\end{equation}
and $a\in(0,\infty)$ is a parameter. 
\end{lem}

\begin{lem}
\label{lem2}
Let $M_p,c_1$ and $g_0$ be nonnegative constants and $p\in (1,2)$. 
Then there exist a positive constant $\lambda>0$ and a monotonically 
decaying function $a(t)>0$,
$a(t)\searrow 0$, such that the following conditions hold:
\begin{align}
\label{6eq1}
\frac{M_p}{a(t)} &\le \frac{1}{2}\bigg{(}\frac{\lambda}{a^q(t)}\bigg{)}^{p-1}\bigg{[}1 - q\frac{|\dot{a}(t)|}{a(t)}\bigg{]},\qquad q(p-1)=1,\\
\label{6eq2}
c_1\frac{|\dot{a}(t)|}{a(t)} &\le \frac{a^q(t)}{2\lambda}\bigg{[}1 - q\frac{|\dot{a}(t)|}{a(t)}\bigg{]},\\
\label{6eq3}
g_0 \frac{\lambda}{a(0)}&< 1.
\end{align}
\end{lem}

\begin{proof}
Choose the function $a(t)$ and positive constants $b,c$ and $d$ such that 
\begin{equation}
\label{eq44}
a(t)= \frac{d}{(c+t)^b},\qquad c \ge \max \bigg{(}1,\, 2bq \bigg{)},\quad 0< b \le p-1, 
\end{equation}
where the constant $d>0$ will be specified later. 
Then
\begin{equation}
\label{eq452}
q\frac{|\dot{a}(t)|}{a(t)} = \frac{q b}{c+t} \le \frac{q b}{c}
\le \frac{1}{2},\qquad \forall t\ge 0.
\end{equation}
Thus, inequality \eqref{6eq1} holds if
\begin{equation}
\label{lamda}
4M_p \le \lambda^{p-1},
\end{equation}
where the relation $q(p-1)=1$ was used.
Choose $\lambda \ge (4M_p)^q$. Then inequality \eqref{6eq1} is satisfied 
for any  $d>0$.

Choose 
\begin{equation}
\label{eq473}
d \ge  \max \bigg{(} g_0\lambda c^b+1,\, (4\lambda c_1 b)^{p-1} \bigg{)}.
\end{equation}
Then inequality \eqref{6eq3} is satisfied. 
From the relations \eqref{eq44} and  
inequalities \eqref{eq473} 
and  \eqref{eq452}, one gets
\begin{equation}
c_1\frac{|\dot{a}(t)|}{a^{q+1}(t)} = \frac{c_1 b}{d^q (c+t)^{1-bq}} 
\le \frac{c_1 b}{d^q}
\le \frac{c_1 b}{4\lambda c_1 b}=\frac {1}{4\lambda}
\le \frac{1}{2\lambda}\bigg{[}1 - q\frac{|\dot{a}(t)|}{a(t)}\bigg{]},
\end{equation}
where inequality \eqref{eq452} was used.
This implies inequality \eqref{6eq2}. Lemma~\ref{lem2} is proved.
\end{proof}

\begin{rem}
\label{rem7}
{\it One can choose $d$ and $\lambda$ so that
the quantity $\frac{a(0)}{\lambda}$ is uniformly bounded as 
$M_p\to\infty$.}

{\rm Indeed, using inequality \eqref{lamda} one can choose
\begin{equation}
\label{eq492}
\lambda = (4M_p)^q.
\end{equation}
Using inequality \eqref{eq473} one can choose
\begin{equation}
\label{eq502}
d =  \max \bigg{(} g_0\lambda c^b+1,\, (4\lambda c_1 b)^{p-1} \bigg{)}.
\end{equation}
It follows from \eqref{eq492} and \eqref{eq502} 
and the assumption $p\in (1,2)$ that 
{\it the quantity $\frac{a(0)}{\lambda} = \frac{d}{\lambda c^b}$
is bounded as $M_p\to\infty$.}

 Indeed, if $M_p\to \infty$,
then $\lambda \to \infty$, because $M_p=\frac {1}{4}\lambda^{p-1}$
and $p>1$.
Moreover, $\lambda^{p-1}<\lambda$ as $\lambda\to \infty$, because
$p-1<1$. Thus, $d=g_0\lambda c^b+1$ as $\lambda\to \infty$,
so $\frac{d}{\lambda c^b}<g_0+1$ as $\lambda\to \infty$.
}
\end{rem}

\begin{proof}[Proof of Theorem~\ref{thm8}]
Denote
\begin{equation}
\label{8eq1}
w(t):= u(t) - V(t),\qquad g(t):=\|w(t)\|,
\end{equation}
where $V(t)$ solves equation \eqref{eq40} with $a=a(t)$. 

Equation  \eqref{eq46} can be rewritten as
\begin{equation}
\label{8eq3}
\dot{w} = -\dot{V} - A^{-1}_{a(t)}[F(u)-F(V)+a(t)w].
\end{equation}
From inequality \eqref{8eq2} one gets
\begin{equation}
\label{8eq4}
F(u) - F(V) + aw = A_{a}(u)w + K,\qquad \|K\| \le M_p \|w\|^p.
\end{equation}
Multiplying \eqref{8eq3} by $w$ and using \eqref{8eq4}, one obtains
\begin{equation}
\label{eq472}
g\dot{g} \le -g^2 +  M_p \|A_{a(t)}^{-1}\|g^{1+p} + \|\dot{V}\|g.
\end{equation}
Multiply equation \eqref{eq40} (with $a=a(t)$) by $V-y$ and use the 
monotonicity of $F$ to get
\begin{equation}
0 = \langle F(V) + a(t)V - F(y), V - y\rangle \ge \langle a(t)V, V - y \rangle. 
\end{equation}
It follows from this inequality that
\begin{equation}
\label{eq49}
 \|V(t)\|\leq \|y\|,\qquad \forall t \ge 0.
\end{equation}
Differentiating equation \eqref{eq40} (with $a=a(t)$) with respect to $t$, 
one gets
\begin{equation}
F'(V)\dot{V} +a(t)\dot{V}+ \dot{a}(t)V = 0.
\end{equation}
This inequality and inequality \eqref{eq49} imply
\begin{equation}
\label{eq512}
\|\dot{V}\| = \|\dot{a}(t)A_{a(t)}^{-1} V\| \le 
\frac{|\dot{a}(t)|}{a(t)}\|y\|, 
\end{equation}
where the estimate $\|A_a^{-1}\|\leq \frac {1} {a}$ was used.
This estimate holds because of the assumption $A\geq 0$.
From inequalities \eqref{eq472} and \eqref{eq512} and 
the relation $g(t)\ge 0$, one gets
\begin{equation}
\dot{g} \le - g(t) + \frac{c_0}{a(t)}g^p + \frac{|\dot{a}(t)|}{a(t)}c_1,
\qquad c_0:= M_p,\quad c_1:=\|y\|.
\end{equation}
This inequality is inequality \eqref{neq1} with
\begin{equation}
\gamma = 1,\quad \alpha(t) = \frac{c_0}{a(t)},\quad \beta(t) = 
c_1\frac{|\dot{a}(t)|}{a(t)}.
\end{equation} 

Let us now apply Corollary~\ref{corrol2} with 
\begin{equation}
\mu(t) = \frac{\lambda}{a^q(t)},\qquad \lambda>0,\quad q(p-1) = 1,
\quad \theta=\frac{1}{2}, 
\end{equation}
where $\lambda$ and $a(t)$ satisfy conditions \eqref{6eq1}--\eqref{6eq3}. 
From inequalities \eqref{6eq1}--\eqref{6eq3} and Corollary~\ref{corrol2} one concludes that
\begin{equation}
\label{eq51}
g(t) < \frac{a(t)}{\lambda}.
\end{equation}
From the triangle inequality one gets
\begin{equation}
\label{eq52}
\|u(t) - y\| \le \|u(t)-V(t)\| + \|V(t)-y\|.
\end{equation}
One has: 
\begin{equation}
\label{eq86}
\|u(t)-V(t)\|=g(t)\leq \frac{a(t)}{\lambda}\leq 
\frac{a(0)}{\lambda}.
\end{equation} 
Using \eqref{eq86}, one gets
\begin{equation}
\label{eq87}
\|V(t)-y\|\leq \|V(t)\|+\|y\|\leq 2\|y\|.
\end{equation} 
Inequalities \eqref{eq52}--\eqref{eq87}, 
imply the following estimates:
\begin{equation}
\label{stay}
\|u(t) - y\| \le \frac{a(t)}{\lambda}+\|V(t)-y\|\leq \frac{a(0)}{\lambda} 
+ 2\|y\|=:R.
\end{equation}
By Remark~\ref{rem7} one can choose $d$ and $\lambda$ such that 
the quantity $\frac{a(0)}{\lambda}$ is uniformly bounded as 
$M_p\to\infty$.
Thus, 
one concludes that $R$ can be chosen independently of $M_p=M_p(R)$. 
Inequality \eqref{stay} implies that the trajectory of $u(t)$ stays 
for all $t\geq 0$ inside a ball $B(y,R)$, where $R$ is a sufficiently 
large fixed number.

Since $a(t)\to 0$ as $t\to\infty$, it follows from the first inequality in 
\eqref{stay} 
and Lemma~\ref{rejectedlem} that
\begin{equation}
\lim_{t\to\infty}\|u(t) -y\|=0,
\end{equation}
where we took into account that $V(t) = V_{a(t)}$. 
Theorem~\ref{thm8} is proved.
\end{proof}

\end{document}